\documentclass[a4paper]{article}

\usepackage{amsmath,amssymb}	

\usepackage{graphicx}

\usepackage{tjm}

\newtheorem{thm}{Theorem}
\newtheorem{lem}{Lemma}
\newtheorem{prop}{Proposition}

\newtheorem{defn}{Definition}
\newtheorem{remark}{Remark}
\newtheorem{claim}{Claim}
\newtheorem{example}{{Example}}

\title[combinatorial Ricci curvature and Lin-Lu-Yau's Ricci Curvature]{Relation between combinatorial Ricci curvature and Lin-Lu-Yau's Ricci Curvature on cell complexes}

\author{Kazuyoshi Watanabe and Taiki Yamada$^{corresponding\ author}$}

\affils{Tohoku University}



\authorname{Kazuyoshi Watanabe}
\address{Mathematical Institute,\\
Tohoku University, Sendai, Miyagi, 985-8578 Japan.}
\email{kazuyoshi.watanabe.q5@dc.tohoku.ac.jp}

\authorname{Taiki Yamada}
\address{Mathematical Institute,\\
Tohoku University, Sendai, Miyagi, 985-8578 Japan.}
\email{mathyamada@dc.tohoku.ac.jp}

\begin{document}

\maketitle
\begin{abstract}
In this paper we compare the combinatorial Ricci curvature on cell complexes and Lin-Lu-Yau's Ricci curvature defined on graphs. On a cell complex, the combinatorial Ricci curvature is introduced by the Bochner-Weitzenb\"{o}ck formula. A cell complex corresponds to a graph such that the vertices are cells and the edges are vectors on the cell complex. We compare these two kinds of Ricci curvatures by the coupling and the Kantorovich duality.
\end{abstract}
\section{Introduction}

In the Riemannian geometry, the curvature plays an important role. Especially the Ricci curvature is studied in geometric analysis on Riemannian manifolds and there are many results on manifolds with non-negative Ricci curvature or with the Ricci curvature bounded below. In the analysis of the Ricci curvature the Bochner-Weitzenb\"{o}ck formula is useful. This formula gives a relation between the curvature tensor and the Hodge Laplacian on smooth differential forms.

There are some definitions of generalized Ricci curvature, one of which is Ollivier's coarse Ricci curvature (see \cite{Ol1}). It is formulated by the 1-Wasserstein distance on a metric space $(X, d)$ with a random walk $m=\left\{m_{x} \right\}_{x \in X}$, where $m_{x}$ is a probability measure on $X$. The coarse Ricci curvature is defined as, for two distinct points $x, y \in X$, 
	\begin{eqnarray*}
	\kappa(x, y) := 1 - \cfrac{W(m_{x}, m_{y})}{d(x, y)},
	\end{eqnarray*}
  where $W$($m_{x}, m_{y}$) is the $1$-Wasserstein distance between $m_{x}$ and $m_{y}$. In 2010, Lin, Lu and Yau \cite{linluyau} modified the definition of Ollivier's coarse Ricci curvature on graphs. They showed some properties about this curvature, such as the Cartesian product and Erd\"{o}s-Renyi's random graph. In 2012, Jost and Liu \cite{Jo2} studied the relation between the Ricci curvature and the local clustering coefficient. Recently, the Ricci curvature on graphs was applied to directed graph \cite{Yamada}, internet topology and so on. In this paper, we call this Ricci curvature {\em the LLY-Ricci curvature}, and denote it by $\kappa$.

On the other hand, a cell complex is studied as a combinatorial space and applied to various works. In a recent work, Forman established some discrete analogues of differential geometry to cell complexes. In \cite{forman-morse}, the discrete Morse theory was established and the Morse inequality, which is the relation between critical cells and the homology of the cell complex, was constructed. The theory has various practical applications in diverse fields of applied mathematics and computer science. By using the discrete vector field the discrete Morse theory is extended to the discrete Novikov-Morse theory \cite{forman-novikov}, and in this theory he defined a differential form on the cell complex. A differential form is not defined as the cochain of the cell complex but as a linear map on the chain of the cell complex. 

 In \cite{kazu}, the first name author introduce the definition of the combinatorial Ricci curvature on a cell complex by the Bochner-Weitzenb\"{o}ck formula on combinatorial differential forms. A combinatorial 1-form $\omega$ has a value at a pair of cells $(\tau >\sigma)$ and we call such a pair $(\tau>\sigma)$ a vector provided that $\sigma$ is a face of $\tau$ and $\operatorname{dim} \sigma +1 =\operatorname{dim} \tau $. For the definition of the covariant derivative $|\nabla \omega|$ of a combinatorial 1-form, we present the parallel vectors, that are called by 0- and 2-neighbor vectors. We define the covariant of a 1-form $|\nabla \omega|$ and the Laplacian for the absolute value of combinatorial 1-form $\Delta^\flat |\omega|^2$ as the difference between the components of parallel vectors. Then we define the combinatorial Ricci curvature by 
 \begin{eqnarray}
\operatorname{Ric} (\omega) (\tau>\sigma) = \langle \Delta \omega, \omega \rangle (\tau>\sigma) -\frac{1}{2} |\nabla \omega|^2  (\tau>\sigma)+ \frac{1}{2} \Delta^\flat |\omega|^2(\tau>\sigma).\nonumber
\end{eqnarray}
  For a graph and for a 2-dimensional cell complex decomposing a closed surface, we have the Gauss-Bonnet theorem for this combinatorial Ricci curvature \cite{kazu}.

For a cell complex $M$, we consider the graph $G_M$ such that the set of vertices of $G_M$ is the set of cells in $M$, and the set of edges of $G_M$ is the set of vectors in $M$. We call $G_M$ the corresponding graph to a cell complex $M$.

We study the relationship between our combinatorial Ricci curvature on $M$ and the LLY-Ricci curvature on $G_M$. The LLY-Ricci curvature is defined by the Wasserstein distance between probability measures. One of our main theorems is stated as follows.

\begin{thm}\label{relate}
Let $\tau^{p+1}$ and $\sigma^{p}$ respectively be a $(p+1)$-cell and a $p$-cell of $M$ such that $\tau>\sigma$. Then we have
\begin{eqnarray}
\kappa (\tau, \sigma) = \frac{\operatorname{Ric}(\tau>\sigma)}{d_\tau \vee d_\sigma}+2\left( \frac{1}{d_\tau \wedge d_\sigma} - \frac{1}{d_\tau \vee d_\sigma} \right) + \frac{d_\tau \wedge d_\sigma}{d_\tau \vee d_\sigma} -1,
\end{eqnarray}
where $d_\tau \wedge d_\sigma = \min \{d_\tau, d_\sigma \}$ and $d_\sigma \vee d_\sigma = \max \{ d_\tau, d_\sigma\}$.
\end{thm}

For the estimate of a lower bound of the LLY-Ricci curvature, we construct the coupling between two probability measures around cells. Since the supports of two probability measures do not intersect each other, this coupling is calculated by combinatorially. Moreover we estimate an upper bound of the LLY-Ricci curvature by using the combinatorial Ricci curvature. From the Kantorovich duality, a 1-Lipschitz function gives the lower bound of the Wasserstein distance between probability measures. Since two bounds coincide with each other, we prove Theorem \ref{relate}.

In section 5, we see that the LLY-Ricci curvature gives the lower bound of the first nonzero eigenvalue of the Laplacian on a cell complex. By applying the Kantorovich duality to the eigenfunction with respect to the first nonzero eigenvalue we estimate the Wasserstein distance between two probability measures. The estimate is useful for computing practical cell complexes. Actually we see the example of a cell complex with the positive combinatorial Ricci curvature, and check the lower estimation of the first nonzero eigenvalue of the Laplacian.

\section*{Acknowledgment} 
The authors thank their supervisors, Professor Takashi Shioya, for his continuous support and providing important comments. They also thank the referee for his/her valuable comments and suggestions.

\section{Definition of combinatorial Ricci curvature}
\subsection{Combinatorial differential form}
In this section, we present a differential form on a cell-complex introduced in \cite{forman-novikov}. Let $M$ be a cell complex. For cells $\sigma$ and $\tau$, we write $\sigma < \tau$ or $\tau > \sigma$ if $\sigma$ is contained in the boundary of $\tau$. First, we define a regular cell complex.
\begin{defn}
We say $M$ is a {\em regular cell complex}, if for each p-cell $\sigma$ of $M$ the characteristic map $h_\sigma: e^p \to M$ maps $e^p$ homeomorphically onto its image, where $e^p$ is a closed ball in the p-dim Euclidean space.
\end{defn}
Throughout the paper, we always assume that $M$ is a regular cell complex. If not, it will be clearly stated.  Let the dimension of $M$ be $n$, and 
\begin{eqnarray}
0 \longrightarrow C_n(M) \overset{\partial}{\longrightarrow} C_{n-1}(M) \overset{\partial}{\longrightarrow} \cdots \overset{\partial}{\longrightarrow} C_0(M) \longrightarrow 0 \nonumber
\end{eqnarray}
be the real cellular chain complex of $M$. We set
\begin{eqnarray}
C_* (M) =\bigoplus_p C_p (M).
\end{eqnarray}
A linear map $\omega : C_*(M) \rightarrow C_*(M)$ is said to be of {\it degree} $d$ if for all $p=1,...,n$,
\begin{eqnarray}
\omega (C_p(M)) \subset C_{p-d} (M).
\end{eqnarray} 
We say that a linear map $\omega$ of degree $d$ is $local$ if, for each $p$ and each oriented $p$-cell $\alpha$, $\omega(\alpha)$ is a linear combination of oriented ($p-d$)-cells that are faces of $\alpha$.
\begin{defn}
For $d\geq0$, we say that a local linear map $\omega : C_*(M) \rightarrow C_*(M)$ of degree $d$ is {\it a combinatorial differential $d$-form}, and we denote the space of combinatorial differential $d$-forms by $\Omega ^d (M)$.
\end{defn}
We define {\it the differential of combinatorial differential forms}
\begin{eqnarray}
d:\Omega^d(M) \rightarrow \Omega^{d+1}(M)
\end{eqnarray}
as follows. For any $\omega \in \Omega^d(M)$ and any $p$-chain $c$, we define $(d\omega) (c) \in C_{p-(d+1)} (M)$ by
\begin{eqnarray}
(d\omega)(c) = \partial (\omega(c)) - (-1)^d \omega (\partial c).
\end{eqnarray}
That is,
\begin{eqnarray}
d\omega = \partial \circ \omega - (-1)^d \omega \circ \partial.
\end{eqnarray}

Let us define an inner product on $C_* (M)$. For any two $p$-cells $\sigma,\sigma'$, we set an inner product as
\begin{eqnarray}
\langle \sigma, \sigma' \rangle = \delta _{\sigma,\sigma'},
\end{eqnarray}
where $\delta _{\sigma,\sigma'}$ is the Kronecker delta, that is, $\delta _{\sigma,\sigma'}=1$ for $\sigma=\sigma'$ and the others are 0. We define the $L^2$ inner product for combinatorial differential forms. For two $d$-forms $u,v$, we set
\begin{eqnarray}
\langle u,v \rangle= \sum_{\sigma} \langle u(\sigma), v(\sigma) \rangle,
\end{eqnarray}
where the sum is taken over all cells $\sigma$ in $M$. \\
Let us consider the adjoint operator of differential with respect to the inner product,
\begin{eqnarray}
d^* : \Omega^d (M) \rightarrow \Omega^{d-1}(M).
\end{eqnarray}
That is, for a $d$-form $u$ and a $(d-1)$-form $v$ we have
\begin{eqnarray}
\langle d^*u,v \rangle=\langle u,dv\rangle.
\end{eqnarray}
\begin{defn}
We define {\it the Laplacian for combinatorial differential forms} by
\begin{eqnarray}
\Delta = d d^* + d^* d.
\end{eqnarray}
\end{defn}

\subsection{Combinatorial function and 1-form on cell complexes}
We realize a combinatorial 0-form as a function. We take a 0-form $f\in \Omega ^0 (M)$, that is,
\begin{eqnarray}
f: C^*(M) \rightarrow C^*(M).
\end{eqnarray}
For any cell $\sigma$, we have
\begin{eqnarray}
f(\sigma) = f_\sigma \sigma,
\end{eqnarray}
and consider $f_\sigma \in \mathbf{R}$ as the value of the function $f$. For a $p$-dimensional cell $\tau$, the derivative of $f$ is 
\begin{eqnarray}
df(\tau) = \sum_{\sigma: \tau > \sigma} (f_\tau -f_\sigma) (-1)^{\tau > \sigma} \sigma,
\end{eqnarray}
where the sum is taken over all $(p-1)$-dimensional cells $\sigma$ that are faces of $\tau$, and $(-1)^{\tau > \sigma}$ is the incidence number between $\tau$ and $\sigma$.

Let $\omega \in \Omega^1 (M)$ be a combinatorial 1-form. For a $p$-dimensional cell $\tau$, we set
\begin{eqnarray}
\omega (\tau) = \sum_{\sigma: \tau > \sigma} \omega^\tau _\sigma(-1)^{\tau > \sigma} \sigma,
\end{eqnarray}
where  the sum is taken over all $(p-1)$-dimensional cells $\sigma$ that are faces of $\tau$, and $(-1)^{\tau > \sigma}$ is the incidence number between $\tau$ and $\sigma$. We call the pair $(\tau>\sigma)$ {\it a vector} provided that a $p$-dimensional cell $\sigma$ is a face of $(p+1)$-dimensional cell $\tau$. We say that $\omega$ has the value $\omega^\tau_\sigma$ at the vector $(\tau >\sigma)$.

For any cell $\sigma$, the dual derivative of $\omega$ is
\begin{eqnarray}
d^* (\omega) (\sigma) &=& \partial^* (\omega (\sigma)) - \omega(\partial^* (\sigma))\\
&=&-\sum_{\tau^{p+1} :\tau>\sigma}  \omega^\tau_\sigma + \sum_{\rho^{p-1}:\rho<\sigma} \omega^\sigma_\rho,
\end{eqnarray}
where the first sum is taken over all $(p+1)$-dimensional cells $\tau$ that have $\sigma$ as a face, and the second sum is over all $(p-1)$-dimensional cells $\rho$ that are the faces of $\sigma$.

Then for any cell $\sigma$, the Laplacian of combinatorial function $f$ is represented by
\begin{eqnarray}
\Delta f (\sigma) &=& d^* df (\sigma)\nonumber\\
&=& -\sum_{\tau^{p+1} :\tau>\sigma} df^\tau_\sigma + \sum_{\rho^{p-1}:\rho<\sigma} df^\sigma_\rho \nonumber\\
&=&- \sum_{\tau^{p+1} ; \tau>\sigma} (f_\tau - f_\sigma) + \sum_{\rho^{p-1}; \sigma>\rho} (f_\sigma - f_\rho).\label{laplacian_for_func}
\end{eqnarray}

\subsection{Combinatorial Ricci curvature}
\begin{defn}
Let $M$ be a regular cell complex. We say that $M$ is {\it quasiconvex} if for every two distinct $(p+1)$-cells $\tau_1$ and $\tau_2$ of $M$, if $\bar{\tau_1} \cap \bar{\tau_2}$ contains a $p$-cell $\sigma$, then $\bar{\tau_1} \cap \bar{\tau_2} =\bar{\sigma}$. In particular this implies that $\bar{\tau_1} \cap \bar{\tau_2}$ contains at most one $p$-cell.
\end{defn}

\begin{figure}[h]
 \begin{minipage}{0.4\hsize}
  \begin{center}
   \includegraphics[width=30mm]{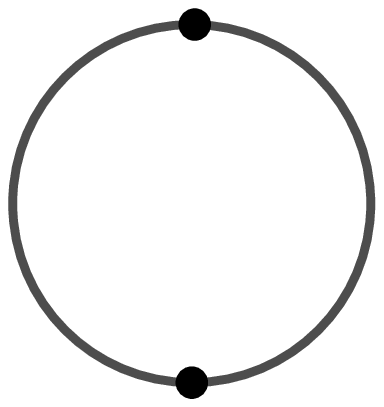}
  \end{center}
  \caption{Non quasiconvex}
  \label{nonquasi}
 \end{minipage}
 \begin{minipage}{0.4\hsize}
  \begin{center}
   \includegraphics[width=30mm]{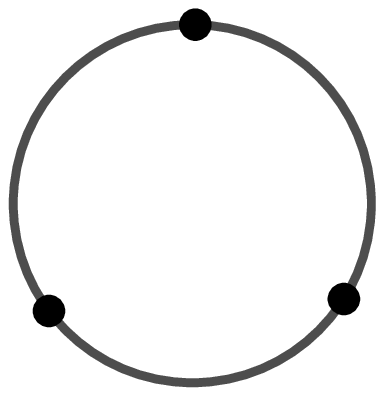}
  \end{center}
  \caption{Quasiconvex}
  \label{quasi}
 \end{minipage}
\end{figure}

Let $M$ be a regular quasiconvex cell complex. We consider ``parallel vectors'', which yield the covariant derivative of a combinatorial 1-form.
\begin{defn}\label{02 vector}
Let $\tau, \sigma$ be two cells of $M$ such that the dimension is $(p+1)$ and $p$ respectively and $\sigma$ is a face of $\tau$. 

We define {\it 0-neighbor vectors of $(\tau>\sigma)$} as the following.
\begin{itemize}
\item vectors $(\tau' > \sigma)$ for $(p+1)$-cells $\tau' \neq \tau$ such that there are no $(p+2)$-cell $\mu$ such that $\mu>\tau, \tau'$.
\item vectors $(\tau>\sigma')$ for $p$-cells $\sigma' \neq \sigma$ such that there are no $(p-1)$-cell $\rho$ such that $\sigma,\sigma'>\rho$.
\end{itemize}
 
 We define {\it 2-neighbor vectors of $(\tau>\sigma)$} as the following.
\begin{itemize}
\item vectors $(\mu > \tau')$ for $(p+1),(p+2)$-cells $\tau'$ and $\mu$ such that $\mu > \tau>\sigma$, $\mu>\tau'>\sigma$ and $\tau \neq \tau'$. 
\item vectors $(\sigma'>\rho)$ for $(p-1),p$-cells $\rho$ and $\sigma'$ such that $\tau>\sigma>\rho$, $\tau>\sigma'>\rho$ and $\sigma \neq \sigma'$.
\end{itemize}
\end{defn}

\begin{remark}
These ideas and names are based on \cite{forman-bochner}. In Figure 3, there exists a 2-form ($\mu \to \sigma$) between $(\tau>\sigma)$ and ($\mu>\tau_1$), so we call the vector ($\mu>\tau_1$) a 2-neighbor vector. In the same way, the vector ($\mu'>\tau'_1$) is also a 2-neighbor vector. On the other hand, there exists a 0-form $(\sigma \to \sigma)$ between $(\tau>\sigma)$ and ($\tau_2>\sigma$), so we call the vector ($\tau_2>\sigma$) a 0-neighbor vector. The vector ($\tau>\sigma_{2}$) is also a 0-neighbor vector.
\end{remark}

\begin{figure}[h]
 \centering
 \includegraphics[scale=0.85]{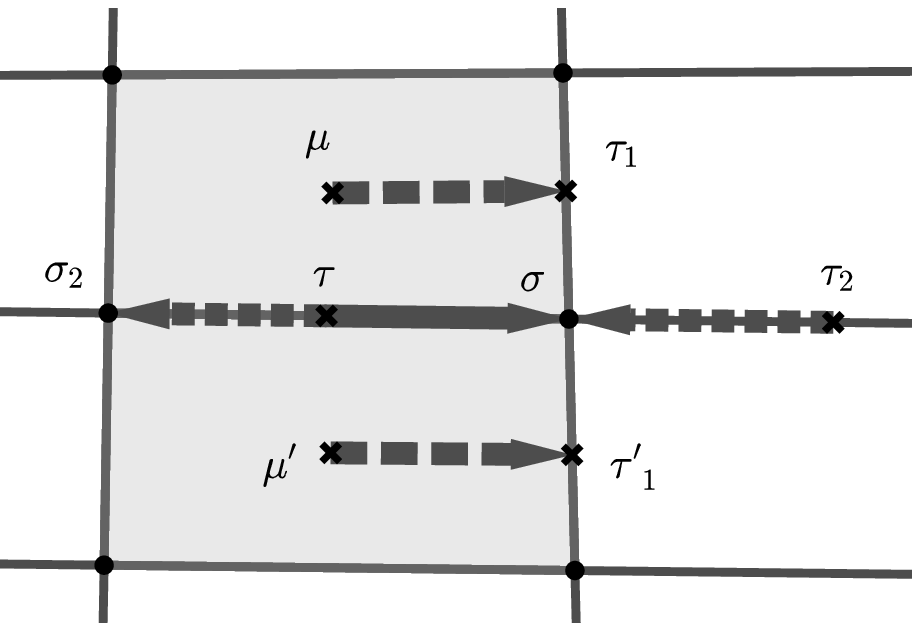}
 \caption{0- and 2- neighbor vectors}
 \label{parallel}
\end{figure}


\begin{defn}
For a combinatorial 1-form $\omega$ on $M$, we define {\it the combinatorial covariant derivative} as
\begin{eqnarray}
|\nabla \omega|^2 (\tau>\sigma) &=& \sum_{(\mu>\tau') ; {\rm 2-neighbor}} (\omega ^\tau _\sigma - \omega^\mu_{\tau'})^2 + \sum_{(\sigma' >\rho); {\rm 2-neighbor } } (\omega^\tau_\sigma - \omega^{\sigma'}_\rho)^2 \nonumber\\
                                               &+& \sum_{(\tau' >\sigma) ; {\rm 0-neighbor}} (\omega^\tau_\sigma +\omega^{\tau'}_\sigma)^2 + \sum_{(\tau >\sigma') ; {\rm 0-neighbor}} (\omega^\tau_\sigma +\omega^\tau_{\sigma'})^2 \nonumber,
\end{eqnarray}
where the sums are taken over all 2-neighbor vectors and 0-neighbor vectors for $(\tau>\sigma)$ respectively. 
\end{defn}

\begin{defn}
For a combinatorial 1-form $\omega$ on $M$, we define {\it the Laplacian of $|\omega|^2$ } as
\begin{eqnarray}
\Delta^\flat |\omega|^2 (\tau>\sigma)  = \sum_{(\mu>\tau') ; {\rm 2-neighbor}}  ((\omega ^\tau _\sigma)^2 - (\omega^\mu_{\tau'})^2) + \sum_{(\sigma' >\rho); {\rm 2-neighbor} }  ((\omega^\tau_\sigma)^2 - (\omega^{\sigma'}_\rho)^2) \nonumber\\
                                               + \sum_{(\tau' >\sigma) ; {\rm 0-neighbor}} ((\omega^\tau_\sigma)^2 -(\omega^{\tau'}_\sigma)^2) + \sum_{(\tau >\sigma') ; {\rm 0-neighbor }} ((\omega^\tau_\sigma)^2 -(\omega^\tau_{\sigma'})^2), \nonumber
\end{eqnarray}
where the sums are taken over all 2-neighbor vectors and 0-neighbor vectors for $(\tau>\sigma)$ respectively. 
\end{defn}

This Laplacian is symmetry for vectors, hence we have
\begin{eqnarray}
\sum_{(\tau>\sigma)} \Delta^\flat |\omega|^2 (\tau>\sigma) =0,
\end{eqnarray}
where the sum is taken over all vectors.
\begin{defn}
For a combinatorial 1-form $\omega$, we define {\it the combinatorial Ricci curvature on a vector $(\tau>\sigma)$} as
\begin{eqnarray}
\operatorname{Ric} (\omega) (\tau>\sigma) = \langle \Delta \omega, \omega \rangle (\tau>\sigma) -\frac{1}{2} |\nabla \omega|^2  (\tau>\sigma)+ \frac{1}{2} \Delta^\flat |\omega|^2(\tau>\sigma).\nonumber
\end{eqnarray}
\end{defn}
\begin{thm}[\cite{kazu}]
Let $M$ be a regular quasiconvex cell-complex, and $(\tau>\sigma)$ a vector on $M$. For a combinatorial 1-form $\omega$ on $M$, the combinatorial Ricci curvature $\operatorname{Ric}(\omega)$ is represented by
\begin{eqnarray}
\operatorname{Ric}(\omega) (\tau>\sigma) = (2- \# \{ {\rm 0 - neighbor\ vector\ of\ } (\tau>\sigma) \})  (\omega ^\tau _\sigma)^2. \label{ricci}
\end{eqnarray}
\end{thm}

We define the combinatorial Ricci curvature on a vector as the one for a unit vector. 

\begin{defn}
Let $M$ be a regular quasiconvex cell-complex, and $(\tau>\sigma)$ a vector on $M$. We define the {\it combinatorial Ricci curvature on a vector} $(\tau>\sigma)$ by
\begin{eqnarray}
\operatorname{Ric} (\tau>\sigma) = (2- \# \{ {\rm 0 - neighbor \ vector \ of\ } (\tau>\sigma) \}).
\end{eqnarray}
\end{defn}

\begin{example}\label{flat}
Let $K^1$ be a 1-dimensional cell complex such that the set of vertices is $\{ v_i \}_{i=-\infty} ^{\infty}$ and the set of edges is $\{ e_i \}_{i=-\infty} ^{\infty}$, and they are related by
\begin{eqnarray}
v_i < e_i > v_{i+1}.
\end{eqnarray}
$K^1$ is homeomorphic to a line. 0-neighbor vectors of $(e_i > v_i)$ are two vectors $(e_{i-1}>v_{i})$ and $(e_i>v_{i+1})$. Then the combinatorial Ricci curvature at the vectors $(e_i > v_i)$ on $K^1$ is 
\begin{eqnarray}
\operatorname{Ric} (e_i > v_i) = \operatorname{Ric} (e_i > v_{i+1})=0.
\end{eqnarray}

We set an $n$-dimensional cell complex $K^n$ by the $n$ times product of $K^1$. Then $K^n$ is homeomorphic to the Euclidean space ${\bf R}^n$, and each $p$-cell of $K^n$ is a $p$-cube. We represent vertices of $K^n$ as integer lattice points $(m_1,...,m_n)$. We consider $p$-cell $\tau=([0,1],...,[0,1] ,0,...,0)$, where first $p$ components are intervals and last $n-p$ components are points. Let $\sigma=([0,1],...,[0,1] ,0,...,0)$ be a $(p-1)$-cell in $\tau$, where first $p-1$ components are intervals and last $n-p+1$ components are points. Then 0-neighbor vectors of $(\tau>\sigma)$ are two vectors $(\tau>([0,1],...,[0,1],1,0,...,0))$ and $(([0,1],...,[0,1],[-1,0],0,...,0)>\sigma)$. Similarly, the number of 0-neighbor vector of any vector $v$ on $K^n$ is 2. Thus, for any vector $v$ on $K^n$ the combinatorial Ricci curvature is
\begin{eqnarray}
\operatorname{Ric} (v) =0.
\end{eqnarray}
Thus $K^n$ is a combinatorially flat space.
\end{example}

\begin{figure}[h]
 \begin{minipage}{0.4\hsize}
  \begin{center}
   \includegraphics[width=40mm]{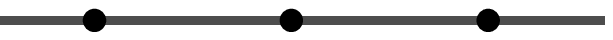}
  \end{center}
  \caption{1-complex $K^1$}
  \label{line}
 \end{minipage}
 \begin{minipage}{0.4\hsize}
  \begin{center}
   \includegraphics[width=40mm]{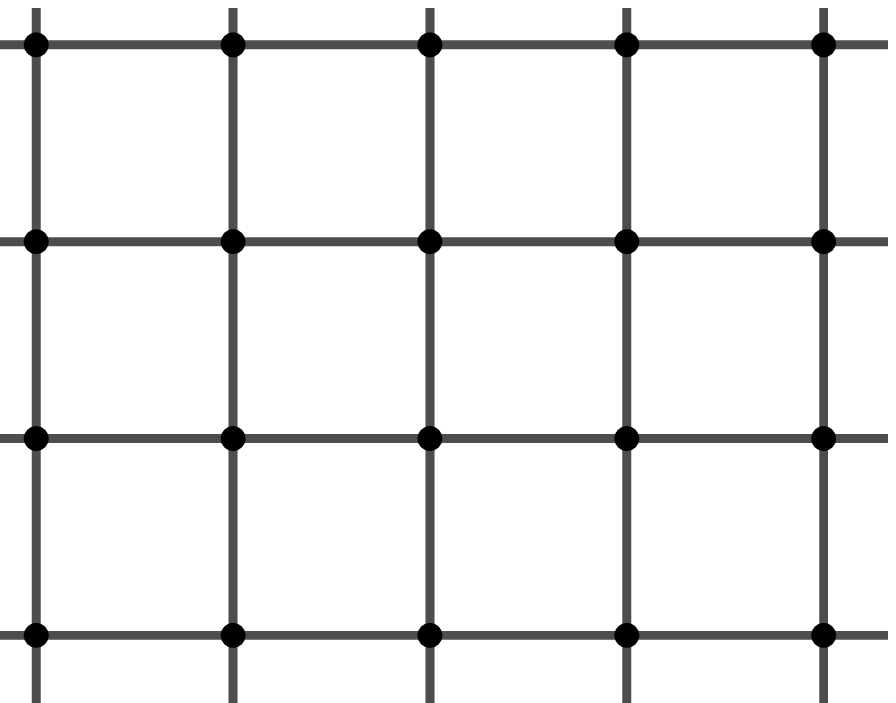}
  \end{center}
  \caption{2-complex $K^2$}
  \label{flat-plane}
 \end{minipage}
\end{figure}

\section{LLY-Ricci curvature on cell complexes}
In \cite{linluyau}, the LLY-Ricci curvature is defined on a graph by using the Wasserstein distance. For a cell complex $M$ we consider the corresponding graph $G_M$, i.e., the set of vertices of $G_M$ is the set of cells in $M$, and the set of edges of $G_M$ is the set of vectors in $M$. 

\begin{figure}[h]
 \begin{minipage}{0.4\hsize}
  \begin{center}
   \includegraphics[width=40mm]{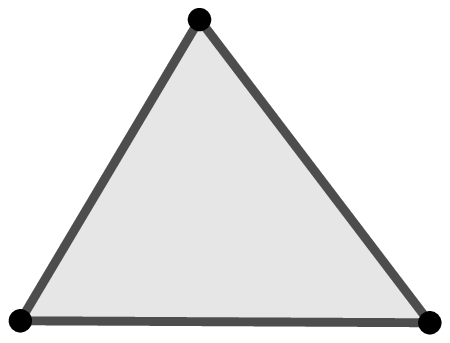}
  \end{center}
  \caption{Cell complex $M$}
  \label{cellc}
 \end{minipage}
 \begin{minipage}{0.4\hsize}
  \begin{center}
   \includegraphics[width=40mm]{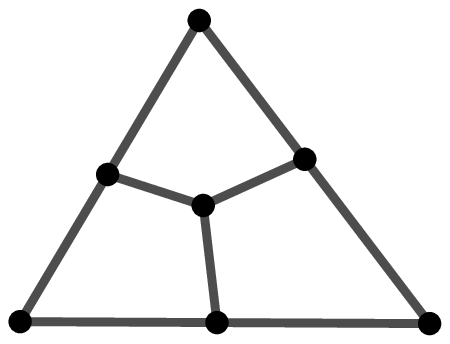}
  \end{center}
  \caption{Corresponding graph $G_M$}
  \label{fig:two}
 \end{minipage}
\end{figure}

Then we define the LLY-Ricci curvature on the graph $G_M$. We denote the set of cells in $M$ by $ S$. We consider the distance and a probability measure $\mu$ on a cell complex $M$ as the one on the corresponding graph $G_M$.

\begin{defn}
	\begin{enumerate}
  	\renewcommand{\labelenumi}{(\arabic{enumi})}
	\item A {\em directed path} between $\sigma \in S$ and $\tau \in S$ is a sequence of edges\\ $\left\{(a_{i}, a_{i+1)} \right\}_{i=0}^{n-1}$, where $a_{0} = \sigma$, $a_{n} = \tau$. We call $n$ the {\em length} of the path.
   	\item The {\em distance} $d(\sigma, \tau)$ between two cells $\sigma, \tau \in S$ is given by the length of a shortest directed path from $\sigma$ to $\tau$.
	\item For any $\sigma \in S$, the {\em neighborhood} of $\sigma$ is defined as
	\begin{eqnarray*}
	\Gamma(\sigma) = \left\{c \in M \mid  c=\tau^{p+1} ~ {\rm with}~ \tau>\sigma ~ {\rm or} ~ c=\rho^{p-1}~ {\rm with}~ \rho<\sigma \right\}.
	\end{eqnarray*}
   	\end{enumerate}
\end{defn}

For any $\alpha \in [0,1]$, and for any $p$-cell $\sigma$ in $M$, we define a probability measure $m^{\alpha}_\sigma$ on $M$ by 
\begin{eqnarray}
  m^{\alpha}_\sigma (c) = \begin{cases}
    \alpha, & {\rm if} \  c=\sigma, \\
    \cfrac{1-\alpha}{d_\sigma }, & {\rm if} \ c=\tau^{p+1} ~ {\rm with}~ \tau>\sigma ~ {\rm or} ~ c=\rho^{p-1}~ {\rm with}~ \rho<\sigma,\\
    0, & \mathrm{otherwise},
  \end{cases}
\end{eqnarray}
where $d_\sigma = \# \Gamma(\sigma).$ We call $d_\sigma$ the {\em degree} of $\sigma$.

\begin{defn}
For two probability measures $\mu$, $\nu$ on $M$, the {\it Wasserstein distance between $\mu$ and $\nu$} is written as
\begin{eqnarray}
W(\mu,\nu)=\inf_A \sum_{\sigma_1, \sigma_2\in S} A(\sigma_1, \sigma_2) d (\sigma_1,\sigma_2),
\end{eqnarray}
where $A: S \times  S \rightarrow [0,1]$ runs over all maps satisfying
\begin{eqnarray}
\begin{cases}
    \sum_{\sigma_2 \in S} A(\sigma_1, \sigma_2) = \mu (\sigma_1), \\
    \sum_{\sigma_1 \in S} A(\sigma_1, \sigma_2) = \nu (\sigma_2).
  \end{cases}
\end{eqnarray}
This map $A$ is called a {\it coupling} between $\mu$ and $\nu$.
\end{defn}

   One of the most important properties of the Wasserstein distance is the Kantorovich-Rubinstein duality as stated as follows.

\begin{prop}[Kantorovich-Rubistein Duality, \cite{Vi2}]
For two probability measures $\mu$, $\nu$ on $M$, the Wasserstein distance between $\mu$ and $\nu$ is written as
\begin{eqnarray}
W(\mu, \nu)=\sup_{f:1-Lip} \sum_{\sigma_1\in S} f(\sigma_1)(\mu(\sigma_1) - \nu (\sigma_1)),
\end{eqnarray}
where the supremum is taken over all functions on $S$ that satisfy $|f(\sigma_1) - f(\sigma_2) | \neq d(\sigma_1,\sigma_2)$ for any cells $\sigma_1$, $\sigma_2$ in $M$, $\sigma_1\neq\sigma_2$.
\end{prop}

\begin{defn}
For $\alpha \in [0,1]$ and for any two cells $\sigma$ and $\sigma'$ of $M$, the {\it $\alpha$-Ricci curvature of $\sigma$ and $\sigma'$} is defined as
\begin{eqnarray}
\kappa_\alpha (\sigma,\sigma')=1- \frac{W(m^\alpha_\sigma, m^\alpha_{\sigma'})}{d(\sigma,\sigma')}.
\end{eqnarray}
\end{defn}

\begin{lem}[\cite{linluyau}]
For any $\alpha \in [0,1]$ and for any two cells $\sigma$ and $\sigma'$ of $M$, we have
\begin{eqnarray}
\kappa_\alpha (\sigma,\sigma') \leq (1-\alpha) \frac{2}{d(\sigma,\sigma')}.
\end{eqnarray}
\end{lem}

\begin{lem}[\cite{linluyau}]
For any two cells $\sigma$ and $\sigma'$ of $M$, the $\alpha$-Ricci curvature $\kappa_\alpha$ is concave in $\alpha \in [0,1]$.
\end{lem}

These two lemmas imply that the function $h(\alpha) =\kappa_\alpha (\sigma,\sigma')/ (1-\alpha) $ is a monotone increasing function in $\alpha$ over $[0,1)$ and bounded. Thus the limit
\begin{eqnarray}
\kappa(\sigma,\sigma') := \lim_{\alpha \rightarrow 1} \frac{ \kappa_\alpha (\sigma,\sigma')}{1-\alpha} 
\end{eqnarray}
exists. We call this limit the {\it LLY-Ricci cuvature} $\kappa(\sigma,\sigma') $ at $(\sigma,\sigma')$ in $M$.

\begin{lem}[\cite{linluyau}]
If $\kappa (\tau, \sigma) \geq \kappa_0$ for any vector $(\tau>\sigma)$ on $M$, then $\kappa (\sigma, \sigma') \geq \kappa_0$ for any pair of cells $(\sigma,\sigma')$.
\end{lem}
We consider the LLY-Ricci curvature only for vectors $(\tau>\sigma)$.


\section{Comparison between two kinds of Ricci curvatures}
In this section we compare the combinatorial Ricci curvature and the LLY-Ricci curvature. We assume that $M$ is a regular quasiconvex cell complex. 

For a $(p+1)$-cell $\tau^{p+1}$ and a $p$-cell $\sigma^{p}$ of $M$ with $\tau>\sigma$, we define the numbers $n_{\tau}$ and $n_{\sigma}$ by
\begin{eqnarray}
n_{\tau} = \# \{ \sigma^p_2 | (\tau>\sigma_2)\in N_0(\tau >\sigma)  \},\\
n_{\sigma} = \# \{ \tau^{p+1}_2 | (\tau_2>\sigma)\in N_0(\tau >\sigma)  \}.
\end{eqnarray}
It holds that $n_\tau + n_\sigma = \# N_0(\tau>\sigma)$. To prove one of the most important properties about the relation between this numbers and degrees, we prepare the following lemma.
\begin{lem}[\cite{lundel}]
\label{lundel}
Let $M$ be a regular cell complex. For any $(p+1)$-cell $\beta$ and $(p-1)$-cell $\gamma$ with $\gamma < \beta$,
\begin{eqnarray}
\# \{\alpha^p| \gamma < \alpha <\beta \} = 2.
\end{eqnarray}
\end{lem}
By using this lemma, we prove the following proposition.
\begin{prop}
\label{relation between cell and degree}
Let $M$ be a regular quasiconvex cell complex. For a $(p+1)$-cell $\tau^{p+1}$ and a $p$-cell $\sigma^{p}$ of $M$ with $\tau>\sigma$, we have
\begin{eqnarray}
d_\tau - n_\tau -1 =d_\sigma - n_\sigma -1 =\#N_2 (\tau>\sigma).
\end{eqnarray}
\end{prop}

\begin{figure}[h]
  \begin{center}
   \includegraphics[scale=0.5]{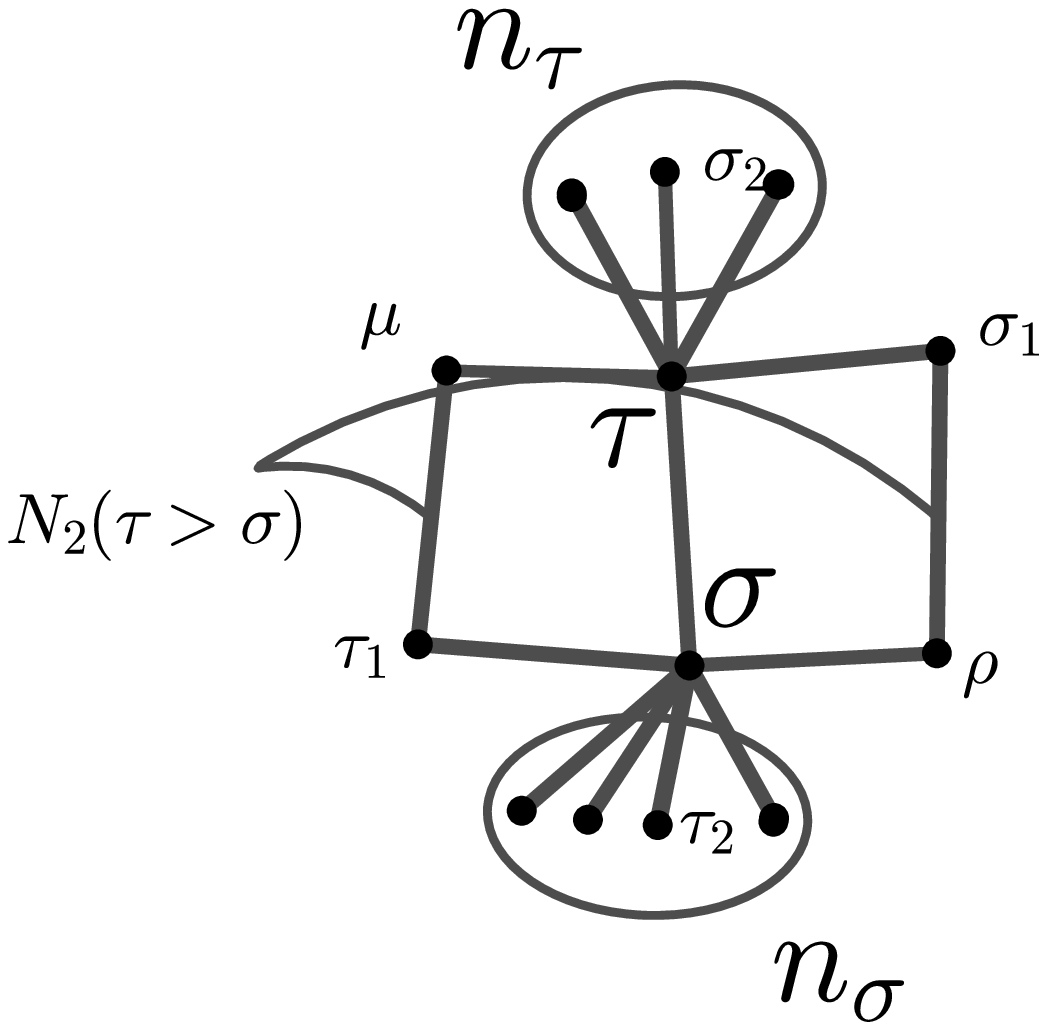}
   \caption{The neighbors of $\sigma$ and $\tau$}
  \end{center}
  \end{figure}
  
\begin{proof}
We prove the equality for the cell $\tau$, i.e., 
\begin{eqnarray}
d_\tau - n_\tau -1 =\#N_2 (\tau>\sigma).
\end{eqnarray}
The equality for the cell $\sigma$ is proved in the same way. The neighborhood of the cell $\tau$ is one of the following four types: 
\begin{enumerate}
 \item The cell $\sigma$.
 \item A p-cell $\sigma_2$ such that $(\tau>\sigma_2)$ are 0-neighbor vectors, i.e., there exist no $(p-1)$-cells $\rho$ such that $\sigma,\sigma_2>\rho$.
 \item A p-cell $\sigma_1$ such that $(\sigma_1>\rho)$ are 2-neighbor vectors, i.e., there exist $(p-1)$-cells $\rho$ such that $\tau>\sigma>\rho$, $\tau>\sigma_1>\rho$ and $\sigma \neq \sigma'$.
 \item A (p+2)-cell $\mu>\tau$.
\end{enumerate}
By the definition of $n_\tau$, the number of p-cells $\sigma_2$ as in 2 is $n_\tau$. Since $M$ is a regular quasiconvex cell complex, for a cell $\sigma_1$ as in 3, there exists a unique $(p-1)$-cell $\rho$ such that $(\sigma_1>\rho)$ is a 2-neighbor vector. By Lemma \ref{lundel}, for a cell $\mu$ as in 4, there exists a unique $(p+1)$-cell $\tau'$ such that ($\mu>\tau'$) is a 2-neighbor vector. Then the total number of cells $\sigma_1$ as in 3 and $\mu$ as in 4 is $\# N_2(\tau>\sigma)$.

Since the number of neighbor cells of $\tau$ is $d_\tau$, we have
\begin{eqnarray}
d_\tau=1+n_\tau +\# N_2(\tau>\sigma).
\end{eqnarray}
This completes the proof of the proposition. 
\end{proof}

For a vector $v$ on $M$, we denote the set of 0-neighbor vectors of $v$ by $N_0(v)$ and the set of 2-neighbor vectors of $v$ by $N_2(v)$ (see Definition \ref{02 vector}). For two real numbers $s$ and $t$, we set
\begin{eqnarray}
s \wedge t = \min \{s, t \},\\
s \vee t = \max \{ s, t\}.
\end{eqnarray}

We have Theorem \ref{relate} from the following two comparisons.
\subsection{Comparison 1}
We first establish the comparison by the coupling.
\begin{thm}
Let $\tau^{p+1}$ and $\sigma^{p}$ be a $(p+1)$-cell and a $p$-cell of $M$ with $\tau>\sigma$. Then we have
\begin{eqnarray}
\kappa (\tau, \sigma) \geq \frac{\operatorname{Ric}(\tau>\sigma)}{d_\tau \vee d_\sigma}+2\left( \frac{1}{d_\tau \wedge d_\sigma} - \frac{1}{d_\tau \vee d_\sigma} \right) + \frac{d_\tau \wedge d_\sigma}{d_\tau \vee d_\sigma} -1.
\end{eqnarray}
\end{thm}

\begin{proof}
By Proposition \ref{relation between cell and degree}, we have 
\begin{eqnarray}
d_\tau - n_\tau -1 =d_\sigma - n_\sigma -1 =\#N_2 (\tau>\sigma).
\end{eqnarray}
Assume that $d_{\sigma} \geq d_{\tau}$, that is
\begin{eqnarray*}
d_{\sigma}= d_\tau \vee d_\sigma,\ d_{\tau}=d_\tau \wedge d_\sigma,
\end{eqnarray*}
we construct the coupling $A$ between $m_{\tau}$ and $m_{\sigma}$. 
So, we define a map  $A: S \times  S \rightarrow [0,1]$ by

\begin{eqnarray*}
\begin{cases}
A(\tau,\sigma)=\alpha, & \\
A(\sigma_2, \tau) =\frac{1}{n_\tau}\frac{1-\alpha}{d_\sigma} &{\rm if}~ (\tau>\sigma^{p}_2) \in N_0(\tau>\sigma),\\
A(\sigma, \tau_2) =\frac{1}{n_\sigma}\frac{1-\alpha}{d_\tau} &{\rm if}~ (\tau^{p+1}_2 > \sigma) \in N_0(\tau>\sigma),\\
A(\mu,\tau_1)=\frac{1-\alpha}{d_\sigma} &{\rm if}~ (\mu^{p+2}>\tau_1^{p+1}) \in N_2(\tau>\sigma),\\
A(\sigma_1,\rho)=\frac{1-\alpha}{d_\sigma} &{\rm if}~ (\sigma^p_1>\rho^{p-1}) \in N_2(\tau>\sigma),\\
A(\mu,\tau_2)=\frac{1}{n_\sigma} \left( \frac{1-\alpha}{d_\tau} - \frac{1-\alpha}{d_\sigma}  \right) &{\rm if}~ (\mu^{p+2}>\tau_1^{p+1}) \in N_2(\tau>\sigma) , \ (\tau^{p+1}_2,\sigma) \in N_0(\tau>\sigma), \\
A(\sigma_1,\tau_2)=\frac{1}{n_\sigma} \left( \frac{1-\alpha}{d_\tau} - \frac{1-\alpha}{d_\sigma}  \right) &{\rm if}~(\sigma^p_1>\rho^{p-1}) \in N_2(\tau>\sigma) , \ (\tau^{p+1}_2,\sigma) \in N_0(\tau>\sigma), \\
A (\sigma_2,\tau_2)=\frac{1}{n_\sigma} \left( \frac{1-\alpha}{d_\tau} - \frac{1}{n_{\tau}}\frac{1- \alpha}{d_{\sigma}} \right) &{\rm if}~ (\tau>\sigma_2^{p}) , \ (\tau_2^{p+1}>\sigma) \in N_0(\tau>\sigma),\\
A(\lambda_1,\lambda_2)=0& \mathrm{otherwise}.
\end{cases}
\end{eqnarray*}
Since we eventually consider the limit as $\alpha$ approaches $1$, we take $\alpha$ so large that $0 < \alpha - \frac{1-\alpha}{d_\sigma} < 1$. 
\begin{claim}
This map $A: S \times  S \rightarrow [0,1]$ is a coupling between $m_\tau$ and $m_{\sigma}$.
\end{claim}

For a $(p+1)$-cell $\tau_2$ with $ (\tau_2>\sigma) \in N_0(\tau>\sigma)$, we obtain
\begin{eqnarray*}
\sum_{\lambda \in S} A(\lambda, \tau_2) &=& \frac{1}{n_\sigma}  \frac{1-\alpha}{d_\tau} + \frac{1}{n_\sigma} \left( \frac{1-\alpha}{d_\tau} - \frac{1-\alpha}{d_\sigma}  \right) (d_\sigma - n_\sigma -1) + \frac{n_\tau}{n_\sigma}\left( \frac{1-\alpha}{d_\tau}-\frac{1}{n_{\tau}}\frac{1-\alpha}{d_{\sigma}}\right)\\
&=& \frac{1}{n_{\sigma}}\frac{1 - \alpha}{d_{\tau}} \left\{ 1 + (d_\tau - n_\tau -1) + n_{\tau} \right\}+ \frac{1}{n_{\sigma}}\frac{1 - \alpha}{d_{\sigma}}\left\{ -(d_\sigma - n_\sigma -1) -1 \right\}\\
&=&\frac{1 - \alpha}{d_\sigma} = m^{\alpha}_\sigma (\tau_2).
\end{eqnarray*}

For a $p$-cell $\sigma_2$ with $(\tau>\sigma^{p}_2) \in N_0(\tau>\sigma)$, we obtain
\begin{eqnarray*}
\sum_{\lambda \in S} A(\sigma_2, \lambda) &=& \frac{1}{n_\tau}  \frac{1-\alpha}{d_\sigma} + \frac{n_\sigma}{n_\sigma}\left( \frac{1-\alpha}{d_\tau}-\frac{1}{n_{\tau}}\frac{1-\alpha}{d_{\sigma}}\right)\\
&=&\frac{1 - \alpha}{d_\tau} = m^{\alpha}_\tau (\sigma_2).
\end{eqnarray*}

For the other cases this is obvious. This implies the claim.

From the definition of the Wasserstein distance, we have
\begin{eqnarray*}
W(m^\alpha_\tau,m^\alpha_\sigma) &\leq&  \sum_{\lambda_1, \lambda_2\in S} A(\lambda_1, \lambda_2) d (\lambda_1,\lambda_2)\\
&=& \alpha + \frac{1-\alpha}{d_{\tau}} + \frac{1-\alpha}{d_\sigma} (d_\tau-n_\tau -1 ) + 3\left( \frac{1-\alpha}{d_\tau} -\frac{1-\alpha}{d_\sigma} \right) (d_\tau - n_\tau -1) \\
&& + \frac{1-\alpha}{d_\sigma} +3 \left( \frac{1-\alpha}{d_\tau} - \frac{1}{n_{\tau}}\frac{1-\alpha}{d_{\sigma}}\right) n_\tau\\
&=& \alpha + \frac{1-\alpha}{d_{\tau}}\left\{1+ 3(d_\tau-n_\tau -1 )+ 3n_{\tau} \right\} - \frac{1-\alpha}{d_{\sigma}}\left\{2(d_\sigma-n_\sigma -1) +2 \right\} \\
&=& \alpha + (1-\alpha) \left\{ 3- \frac{2}{d_\tau} -\frac{(d_\sigma - n_\sigma) + (d_\tau  -n_\tau)}{d_\sigma} \right\}.
\end{eqnarray*}
By the definition of $n_{\sigma}$ and $n_{\tau}$, we obtain
\begin{eqnarray*}
W(m^\alpha_\tau,m^\alpha_\sigma) &\leq& \alpha + (1-\alpha) \left(3- \frac{2}{d_\tau} -\frac{d_\sigma +d_\tau - \# N_0 (\tau>\sigma)}{d_\sigma} \right)\\
&=& \alpha + (1-\alpha) \left\{3- \frac{2}{d_\tau} -\frac{d_\sigma +d_\tau - (2-\operatorname{Ric}(\tau>\sigma))}{d_\sigma} \right\}\\
&=& \alpha + (1-\alpha) \left\{- \frac{\operatorname{Ric}(\tau>\sigma)}{d_\sigma} - 2 \left( \frac{1}{d_\tau} - \frac{1}{d_\sigma} \right) -\frac{d_\tau}{d_\sigma} +2 \right\}.
\end{eqnarray*}
This implies that
\begin{eqnarray*}
\kappa_\alpha (\tau,\sigma) &=& 1-W(m^\alpha_\tau,m^\alpha_\sigma) \\
&\geq& (1-\alpha) \left\{ \frac{\operatorname{Ric}(\tau>\sigma)}{d_{\sigma}}  + 2 \left( \frac{1}{d_\tau} - \frac{1}{d_\sigma} \right) +\frac{d_\tau}{d_\sigma} -1 \right\}.
\end{eqnarray*}
Thus we have
\begin{eqnarray*}
\kappa(\tau,\sigma) &=& \lim_{\alpha \rightarrow 1} \frac{ \kappa_\alpha (\tau,\sigma)}{1-\alpha} \\
&\geq&   \frac{\operatorname{Ric}(\tau>\sigma)}{d_{\sigma}}  + 2 \left( \frac{1}{d_\tau} - \frac{1}{d_\sigma} \right) +\frac{d_\tau}{d_\sigma} -1  .
\end{eqnarray*}
　In the case of $d_{\tau} > d_\sigma$, we consider the following map $A'$.
\begin{eqnarray*}
\begin{cases}
A'(\tau,\sigma)=\alpha, & \\
A'(\sigma_2, \tau) =\frac{1}{n_\tau}\frac{1-\alpha}{d_\sigma} &{\rm if}~ (\tau>\sigma^{p}_2) \in N_0(\tau>\sigma),\\
A'(\sigma, \tau_2) =\frac{1}{n_\sigma}\frac{1-\alpha}{d_\tau} &{\rm if}~ (\tau^{p+1}_2 > \sigma) \in N_0(\tau>\sigma),\\
A'(\mu,\tau_1)=\frac{1-\alpha}{d_\tau} &{\rm if}~ (\mu^{p+2}>\tau_1^{p+1}) \in N_2(\tau>\sigma),\\
A'(\sigma_1,\rho)=\frac{1-\alpha}{d_\tau} &{\rm if}~ (\sigma^p_1>\rho^{p-1}) \in N_2(\tau>\sigma),\\
A'(\sigma_2,\tau_1)=\frac{1}{n_\tau} \left( \frac{1-\alpha}{d_\sigma} - \frac{1-\alpha}{d_\tau}  \right) &{\rm if}~ (\mu^{p+2}>\tau_1^{p+1}) \in N_2(\tau>\sigma) , \ (\tau^{p+1}_2,\sigma) \in N_0(\tau>\sigma), \\
A'(\sigma_2,\sigma_1)=\frac{1}{n_\tau} \left( \frac{1-\alpha}{d_\sigma} - \frac{1-\alpha}{d_\tau}  \right)&{\rm if}~(\sigma^p_1>\rho^{p-1}) \in N_2(\tau>\sigma) , \ (\tau^{p+1}_2,\sigma) \in N_0(\tau>\sigma), \\
A' (\sigma_2,\tau_2)=\frac{1}{n_\tau} \left( \frac{1-\alpha}{d_\sigma} - \frac{1}{n_{\sigma}}\frac{1- \alpha}{d_{\tau}} \right) &{\rm if}~ (\tau>\sigma_2^{p}) , \ (\tau_2^{p+1}>\sigma) \in N_0(\tau>\sigma),\\
A'(\lambda_1,\lambda_2)=0& \mathrm{otherwise}.
\end{cases}
\end{eqnarray*}
This map $A'$ is also a coupling between $m_\tau$ and $m_{\sigma}$. So, we obtain
\begin{eqnarray*}
W(m^\alpha_\tau,m^\alpha_\sigma) &\leq&  \sum_{\lambda_1, \lambda_2\in S} A'(\lambda_1, \lambda_2) d (\lambda_1,\lambda_2)\\
&=& \alpha + (1-\alpha) \left\{ 3- \frac{2}{d_\sigma} -\frac{(d_\sigma - n_\sigma) + (d_\tau  -n_\tau)}{d_\tau} \right\},
\end{eqnarray*}
which implies
\begin{eqnarray*}
\kappa(\tau,\sigma) \geq   \frac{\operatorname{Ric}(\tau>\sigma)}{d_{\tau}}  + 2 \left( \frac{1}{d_\sigma} - \frac{1}{d_\tau} \right) +\frac{d_\sigma}{d_\tau} -1  .
\end{eqnarray*}
This completes the proof of the theorem. 
\end{proof}

\subsection{Comparison 2}
Secondly we establish a comparison by the Kantorovich-Rubistein duality.
\begin{lem}
\label{odd}
Let $M$ be a regular cell complex. Then there is no odd cycle in $G_M$.
\end{lem}

\begin{proof}
Let $\sigma_0,...,\sigma_{2n-1}=\sigma_0$ be a cycle in $G_M$. We assume that $\operatorname{dim} \sigma_0$ is even. Then the dimension of $\sigma_1$ is odd. By continuing this process, $\operatorname{dim} \sigma_{2n-2}$ is even. This is a contradiction to
\begin{eqnarray}
| \operatorname{dim} \sigma_{2n-2} - \operatorname{dim} \sigma_{2n-1} |=1.
\end{eqnarray}
If the dimension of the $\sigma_0$ is odd, then the similar argument is constructed. The proof is completed.
\end{proof}

\begin{thm}\label{copa2}
Let $\tau^{p+1}$ and $\sigma^{p}$ be a $(p+1)$-cell and a $p$-cell such that $\tau>\sigma$. Then we have
\begin{eqnarray}
\kappa (\tau, \sigma) \leq\frac{\operatorname{Ric}(\tau>\sigma)}{d_\tau \vee d_\sigma}+2\left( \frac{1}{d_\tau \wedge d_\sigma} - \frac{1}{d_\tau \vee d_\sigma} \right) + \frac{d_\tau \wedge d_\sigma}{d_\tau \vee d_\sigma} -1.
\end{eqnarray}
\end{thm}

\begin{proof}
We assume that $d_\sigma \geq d_\tau$. Similarly as in the previous theorem, we set the numbers $n_{\tau}$ and $n_{\sigma}$ by
\begin{eqnarray}
n_{\tau} = \# \{ \sigma^p_2 | (\tau>\sigma_2)\in N_0(\tau >\sigma)  \},\\
n_{\sigma} = \# \{ \tau^{p+1}_2 | (\tau_2>\sigma)\in N_0(\tau >\sigma)  \}.
\end{eqnarray}
We define the function $f : \Gamma(\sigma) \cup \Gamma(\tau) \rightarrow {\bf R}$ by
\begin{eqnarray*}
\begin{cases}
f(\tau)=1, &\\
f(\sigma)=0, &\\
f(\mu)=2, \ f(\tau_1)=1 &{\rm if}~ (\mu^{p+2}>\tau_1^{p+1}) \in N_2(\tau>\sigma),\\
f(\sigma_1)=2, \ f(\rho)=1 &{\rm if}~  (\sigma^p_1>\rho^{p-1}) \in N_2(\tau>\sigma),\\
f(\sigma_2)=2 &{\rm if}~ (\tau>\sigma_2^p)\in N_0(\tau>\sigma),\\
f(\tau_2)=-1 &{\rm if}~ (\tau_2^{p+1}>\sigma)\in N_0(\tau>\sigma).\\
\end{cases}
\end{eqnarray*}
For $(\tau>\sigma_2^p)\in N_0(\tau>\sigma)$ and $(\tau_2^{p+1}>\sigma)\in N_0(\tau>\sigma)$, we see $|f(\sigma_2) - f(\tau_2) |=3$. The distance between the two cells $\sigma_2$ and $\tau_2$ is 3, since there is no 5-cycles. By Lemma \ref{odd}, it is clear that $f$ is an 1-Lipschitz function over $\Gamma(\sigma) \cup \Gamma(\tau)$ with respect to the graph distance on $G_{M}$ so that $f$ can be extended to an 1-Lipschitz function over $S$.\\
　From the Kantorovich-Rubistein duality, we have
\begin{eqnarray*}
W(m^\alpha_\tau,m^\alpha_\sigma) &\geq&  \sum_{\lambda\in S} f(\lambda)(m^\alpha_\tau(\lambda) - m^\alpha_{\sigma} (\lambda))\\
&=& 1 \cdot (\alpha - \frac{1-\alpha}{d_\sigma}) + 2(d_{\tau} -1 ) \frac{1-\alpha}{d_\tau} - (d_\sigma - n_\sigma -1)\frac{1-\alpha}{d_\sigma} + n_\sigma \frac{1-\alpha}{d_\sigma}\\
&=& \alpha + 2(d_{\tau} -1 ) \frac{1-\alpha}{d_\tau} - \left( d_\sigma - 2n_\sigma \right)\frac{1 - \alpha}{d_{\sigma}} \\
&=& \alpha + 2 (1 - \alpha) - \frac{2(1 - \alpha)}{d_\tau} + (1 - \alpha) - 2  \left( d_\sigma - n_\sigma \right) \frac{1 - \alpha}{d_{\sigma}} \\
&=& \alpha + (1-\alpha) \left\{ 3- \frac{2}{d_\tau} -\frac{(d_\sigma - n_\sigma) + (d_\tau  -n_\tau)}{d_\sigma} \right\}\\
&=& \alpha + (1 - \alpha) \left\{ 3 - \frac{2}{d_\tau} - \frac{d_\sigma + d_\tau - (2-\operatorname{Ric}(\tau>\sigma))}{d_\sigma} \right\}\\
&=& \alpha + (1 - \alpha) \left\{ - \frac{\operatorname{Ric}(\tau>\sigma)}{d_\sigma} + 2\left( \frac{1}{d_\sigma} - \frac{1}{d_\tau} \right) - \frac{d_\tau}{d_\sigma} +2 \right\}.
\end{eqnarray*}

This implies that
\begin{eqnarray*}
\kappa_\alpha (\tau,\sigma) &=& 1-W(m^\alpha_\tau,m^\alpha_\sigma) \\
&\leq& (1-\alpha) \left\{ \frac{\operatorname{Ric}(\tau>\sigma)}{d_{\sigma}} + 2\left( \frac{1}{d_\tau} - \frac{1}{d_\sigma} \right) + \frac{d_\tau}{d_\sigma} -1  \right\}.
\end{eqnarray*}

Thus we have
\begin{eqnarray*}
\kappa(\tau,\sigma) &=& \lim_{\alpha \rightarrow 1} \frac{ \kappa_\alpha (\tau,\sigma)}{1-\alpha} \\
&\leq&\frac{\operatorname{Ric}(\tau>\sigma)}{d_{\sigma}} + 2\left( \frac{1}{d_\tau} - \frac{1}{d_\sigma} \right) + \frac{d_\tau}{d_\sigma} -1  .
\end{eqnarray*}

In the case of $d_\tau \geq d_\sigma$, we consider the following function $g$.
\begin{eqnarray*}
\begin{cases}
g(\tau)=0, &\\
g(\sigma)=1, &\\
g(\mu)=1, \ f(\tau_1)=2 &{\rm if}~ (\mu^{p+2}>\tau_1^{p+1}) \in N_2(\tau>\sigma),\\
g(\sigma_1)=1, \ f(\rho)=2 &{\rm if}~  (\sigma^p_1>\rho^{p-1}) \in N_2(\tau>\sigma),\\
g(\sigma_2)=-1 &{\rm if}~ (\tau>\sigma_2^p)\in N_0(\tau>\sigma),\\
g(\tau_2)=2 &{\rm if}~ (\tau_2^{p+1}>\sigma)\in N_0(\tau>\sigma).\\
\end{cases}
\end{eqnarray*}
Since the function $g$ is also 1-Lipschitz function, by the Kantorovich-Rubinstein duality, we obtain 
\begin{eqnarray*}
W(m^\alpha_\tau,m^\alpha_\sigma) &\geq&  \sum_{\lambda\in S} g(\lambda)(m^\alpha_\tau(\lambda) - m^\alpha_{\sigma} (\lambda))\\
&=& \alpha + (1 - \alpha) \left\{ - \frac{\operatorname{Ric}(\tau>\sigma)}{d_\tau} + 2\left( \frac{1}{d_\tau} - \frac{1}{d_\sigma} \right) - \frac{d_\sigma}{d_\tau} +2 \right\},
\end{eqnarray*}
which implies
\begin{eqnarray*}
\kappa(\tau,\sigma) \leq   \frac{\operatorname{Ric}(\tau>\sigma)}{d_{\tau}}  + 2 \left( \frac{1}{d_\sigma} - \frac{1}{d_\tau} \right) +\frac{d_\sigma}{d_\tau} -1  .
\end{eqnarray*}

This completes the proof of the theorem.
\end{proof}

\begin{remark}
For a cell complex $M$, if the graph $G_M$ is a $d$-regular graph, i.e., the degrees of all cells of $M$ are constant $d$, then Theorem \ref{relate} yields
\begin{eqnarray}
\frac{\operatorname{Ric}(\tau > \sigma)}{d} = \kappa(\tau, \sigma)
\end{eqnarray}
for any vector $(\tau>\sigma)$.
For the combinatorially flat space $K^n$ that decomposes the $n$-dimensional Euclid space ${\bf R}^n$ by $n$-cubes (see Example \ref{flat}), we have
\begin{eqnarray}
\operatorname{Ric}(\tau > \sigma)=0
\end{eqnarray}
for any vector $(\tau>\sigma)$ and the degree of any cell is $2n$. Therefore we have
\begin{eqnarray}
\kappa(\tau, \sigma)=0
\end{eqnarray}
for any two cells $\tau$ and $\sigma$ of $K^n$.
\end{remark}

\section{The estimate of the first non-zero Laplacian eigenvalue of a cell complex}
In this section we would like to obtain the estimate of the first non-zero Laplacian eigenvalue of a cell complex by the LLY-Ricci curvature. The Laplacian of a cell complex is represented as follows. From the equation \eqref{laplacian_for_func}, for a function $f\in \Omega^0 (M)$, we have
\begin{eqnarray}
\Delta f (\sigma^p) &=& - \sum_{\tau^{p+1} ; \tau>\sigma} (f(\tau) - f(\sigma)) + \sum_{\rho^{p-1}; \sigma>\rho} (f(\sigma) - f(\rho))\\
 &=& d_\sigma f(\sigma) - \sum_{\tau^{p+1} ; \tau>\sigma} f(\tau) - \sum_{\rho^{p-1}; \sigma>\rho}f(\rho).
\end{eqnarray}
This Laplacian is equal to the non normalized Laplacian on the graph $G_M$.\\
　For the LLY-Ricci curvature, the Myers' type theorem for a graph is proved in \cite{linluyau}.
\begin{thm}[\cite{linluyau}]\label{myers-g}
Suppose that $\kappa (\tau>\sigma) \geq \kappa >0$ for any vector $(\tau>\sigma)$ on $M$ and for a real number $\kappa > 0$. Then the diameter of a cell complex $M$ is bounded as follows:
\begin{eqnarray}
\operatorname{diam}(M)\leq \frac{2}{\kappa}.
\end{eqnarray}
\end{thm}
\noindent By using Theorem \ref{myers-g}, we obtain the result about the estimate of the first non-zero Laplacian eigenvalue of a cell complex. 

\begin{thm}\label{estimate}
Let $M$ be a finite cell complex and $\lambda_1$ the first non-zero Laplacian eigenvalue of $M$. If $\kappa(\tau>\sigma)\geq \kappa >0 $ for any vector $(\tau>\sigma)$ on $M$, then we have
\begin{eqnarray}
\lambda_1 \geq \frac{d_{\max} d_{\min} \kappa^2}{\kappa d_{\max} + 2 (d_{\max} - d_{\min})},
\end{eqnarray}
where $d_{\max} = \max_{\sigma \in S} d_\sigma$ and $\ d_{\min} = \min_{\sigma \in S} d_\sigma$.
\end{thm}

\begin{proof}
Let $f$ be an eigenfunction with respect to $\lambda_1$ on $M$. Without loss of generality we assume that $f$ is a 1-Lipschitz function and 
\begin{eqnarray}
\sup_{\lambda\neq \lambda' } \frac{|f(\lambda) - f(\lambda')| }{d(\lambda , \lambda')}=1.
\end{eqnarray}
Then there exists a vector $(\tau>\sigma)$ on $M$ with
\begin{eqnarray}
f(\tau) - f(\sigma)=1,
\end{eqnarray}
by changing the sign of $f$ if necessary. We assume $d_\tau \leq d_\sigma$ without loss of generality. Then we have
\begin{eqnarray}
\Delta f (\sigma^p)  &=& d_\sigma f(\sigma) - \sum_{\tau^{p+1} ; \tau>\sigma} f(\tau) - \sum_{\rho^{p-1}; \sigma>\rho}f(\rho)\\
&=& \lambda_1 f(\sigma),
\end{eqnarray}
and
\begin{eqnarray}
 \sum_{\tau^{p+1} ; \tau>\sigma} f(\tau) + \sum_{\rho^{p-1}; \sigma>\rho}f(\rho) = (d_\sigma - \lambda_1)f(\sigma).
\end{eqnarray}
The Wasserstein distance is estimated by
\begin{eqnarray*}
W(m^\alpha_\tau,m^\alpha_\sigma) &\geq&  \sum_{\lambda\in S} f(\lambda)(m^\alpha_\tau(\lambda) - m^\alpha_{\sigma} (\lambda))\\
&=& \sum_{\mu; \mu>\tau} f(\mu) \frac{1-\alpha}{d_\tau} + \sum_{\sigma_1; \tau>\sigma_1}f(\sigma_1) \frac{1-\alpha}{d_\tau} + f(\tau)\alpha\\
         &&- \sum_{\tau_1; \tau_1 >\sigma} f(\tau_1) \frac{1-\alpha}{d_\sigma} - \sum_{\rho;\sigma>\rho} f(\rho)\frac{1-\alpha}{d_\sigma} - f(\sigma) \alpha \\
&=& (d_\tau - \lambda_1) f(\tau) \frac{1-\alpha}{d_\tau} + f(\tau)\alpha - (d_\sigma - \lambda_1)f(\sigma)\frac{1-\alpha}{d_\sigma} - f(\sigma)\alpha\\
&=& (f(\tau) - f(\sigma)) - (1-\alpha) \lambda_1 \left( \frac{f(\tau)}{d_\tau} - \frac{f(\sigma)}{d_\sigma} \right) \\
&=& 1 + (1 - \alpha) \lambda_1 \left( -\frac{1}{d_\tau} + \left( \frac{1}{d_\sigma} - \frac{1}{d_\tau} \right) f(\sigma) \right).
\end{eqnarray*}
Since $f$ is a 1-Lipschitz function, we have 
\begin{eqnarray}
\max_\nu f(\nu) - \min_\nu f(\nu) \leq \operatorname{diam}(M).
\end{eqnarray}
Since $f$ is an eigenfunction of a non-zero value $\lambda_1$, the function $f$ is orthogonal to the constant function and $\min_\nu f(\nu)$ is negative. Therefore, we obtain
\begin{eqnarray}
\max_\nu f(\nu)  \leq \operatorname{diam}(M).
\end{eqnarray}
From Theorem \ref{myers-g} we have
\begin{eqnarray}
\max_\nu f(\nu) \leq \frac{2}{\kappa}.
\end{eqnarray}
This yields
\begin{eqnarray*}
W(m^\alpha_\tau,m^\alpha_\sigma) &\geq& 1 + (1 - \alpha) \lambda_1 \left( -\frac{1}{d_\tau} + \left( \frac{1}{d_\sigma} - \frac{1}{d_\tau} \right) \frac{2}{\kappa} \right)\\
&\geq&  1 + (1 - \alpha) \lambda_1 \left(- \frac{1}{d_{\min}} + \left( \frac{1}{d_{\max}} - \frac{1}{d_{\min}} \right) \frac{2}{\kappa} \right).
\end{eqnarray*}
Then the $\alpha$-Ricci curvature is
\begin{eqnarray*}
\kappa_\alpha (\tau,\sigma) &=& 1-W(m^\alpha_\tau,m^\alpha_\sigma) \\
&\leq& (1-\alpha)\lambda_1 \left( \frac{1}{d_{\min}} + \left( \frac{1}{d_{\min}} - \frac{1}{d_{\max}} \right) \frac{2}{\kappa} \right).
\end{eqnarray*}
Thus we have
\begin{eqnarray*}
\kappa(\tau,\sigma) &=& \lim_{\alpha \rightarrow 1} \frac{ \kappa_\alpha (\tau,\sigma)}{1-\alpha} \\
&\leq&\lambda_1 \left( \frac{1}{d_{\min}} + \left( \frac{1}{d_{\min}} - \frac{1}{d_{\max}} \right) \frac{2}{\kappa} \right).
\end{eqnarray*}
This implies that
\begin{eqnarray*}
\lambda_1 \geq \frac{\kappa}{\frac{1}{d_{\min}} + \left( \frac{1}{d_{\min}} - \frac{1}{d_{\max}} \right) \frac{2}{\kappa}}\\
\geq \frac{d_{\max} d_{\min} \kappa^2}{\kappa d_{\max} + 2 (d_{\max} - d_{\min})}.
\end{eqnarray*}
The proof is completed
\end{proof}

\begin{example}
Let $C^n$ be the boundary of the $(n+1)$-simplex. We set the vertices $v_0,...,v_{n+1}$ and represent $p$-cell by $[v_{i_0}, v_{i_1}, .., v_{i_p}]$. For an edge $[v_{i_0},v_{i_1}]$ and a vertex $v_{i_0}$, 0-neighbor of the vector $([v_{i_0},v_{i_1}]>v_{i_0})$ is the vector $([v_{i_0},v_{i_1}]>v_{i_1})$. We obtain
\begin{eqnarray}
\operatorname{Ric} ([v_{i_0},v_{i_1}]>v_{i_0}) =1.
\end{eqnarray}
The adjacent cell of $v_{i_0}$ are $[v_{i_0}, v_{j}]$ for $j=0,...,i_0-1,i_0+1,...,n+1$. The adjacent cell of $[v_{i_0},v_{i_1}]$ are two vertices, $v_{i_0}$ and $v_{j_0}$, and 2-cells $[v_{i_0}, v_{i_1}, v_{j}]$ for $j=0,...,n+1$ except for $i_0,i_1$. From Theorem \ref{relate} we have
\begin{eqnarray}
\kappa([v_{i_0},v_{i_1}]>v_{i_0}) &=& \frac{1}{n+2} + 2\left( \frac{1}{n+1} - \frac{1}{n+2} \right) + \frac{n+1}{n+2}-1 \nonumber \\
&=& \frac{2}{(n+1)(n+2)}.
\end{eqnarray}
For a $n$-cell $[v_{i_0}, v_{i_1} .., v_{i_n}]$ and a $(n-1)$-cell $[v_{i_0}, v_{i_1}, ..., v_{i_{j-1}}, v_{i_{j+1}} .., v_{i_n}]$, 0-neighbor of the vector $([v_{i_0}, v_{i_1} .., v_{i_n}]>[v_{i_0}, v_{i_1}, ..., v_{i_{j-1}}, v_{i_{j+1}} .., v_{i_n}])$ is the vector $([v_{i_0}, v_{i_1}, ..., v_{i_{j-1}}, v_{i_{j+1}} .., v_{i_{n+1}}]>[v_{i_0}, v_{i_1}, ..., v_{i_{j-1}}, v_{i_{j+1}} .., v_{i_{n}}])$, where $v_{i_{n+1}}$ is the vertex that is not included in $[v_{i_0}, v_{i_1} .., v_{i_n}]$. We obtain
\begin{eqnarray}
\operatorname{Ric}([v_{i_0}, v_{i_1} .., v_{i_n}]>[v_{i_0}, v_{i_1}, ..., v_{i_{j-1}}, v_{i_{j+1}} .., v_{i_n}]) =1.
\end{eqnarray}
The adjacent cell of $[v_{i_0}, v_{i_1} .., v_{i_n}]$ are $[v_{i_0}, v_{i_1}, ..., v_{i_{k-1}}, v_{i_{k+1}} .., v_{i_{n+1}}]$ for $k=0,...,n+1$. The adjacent cell of $[v_{i_0}, v_{i_1}, ..., v_{i_{j-1}}, v_{i_{j+1}} .., v_{i_{n}}]$ are two $n$-cells, $[v_{i_0}, v_{i_1} .., v_{i_n}]$ and $[v_{i_0}, v_{i_1}, ..., v_{i_{j-1}}, v_{i_{j+1}} .., v_{i_{n+1}}]$, and $(n-2)$-cells that do not include 2 vertices $v_{i_j},v_{i_{n+1}}$. From Theorem \ref{relate} we have
\begin{eqnarray}
\kappa([v_{i_0}, v_{i_1} .., v_{i_n}]&>&[v_{i_0}, v_{i_1}, ..., v_{i_{j-1}}, v_{i_{j+1}} .., v_{i_n}]) \nonumber \\
&=& \frac{1}{n+2} + 2\left( \frac{1}{n+1} - \frac{1}{n+2} \right) + \frac{n+1}{n+2}-1\nonumber\\
&=& \frac{2}{(n+1)(n+2)}.
\end{eqnarray}

For other vector $v$ in $C^n$, 0-neighbor vectors do not exist. We obtain
\begin{eqnarray}
\operatorname{Ric} (v) =2.
\end{eqnarray}
In a similar way to before cases, for $p=1,...,n-1$, the number of the adjacent cell of a $p$-cell is $n+2$. Then we have
\begin{eqnarray}
\kappa (v) =\frac{2}{n+2}.
\end{eqnarray}
Thus the lower bound of the LLY-Ricci curvature on $C^n$ is
\begin{eqnarray}
\kappa=\frac{2}{(n+1)(n+2)}>0.
\end{eqnarray}
From Theorem \ref{estimate}, we have the inequality for the first non-zero eigenvalue of $C^n$,
\begin{eqnarray}
\lambda_1 &\geq&\frac{(n+1)(n+2) \left( \frac{2}{(n+1)(n+2)} \right)^2 }{\frac{2}{(n+1)(n+2)} (n+2) +2(n+1)(n+2) }\\
&=& \frac{2}{(n+2)^2}.
\end{eqnarray}
\end{example}

\begin{figure}[tbh]
  \begin{center}
   \includegraphics[width=70mm]{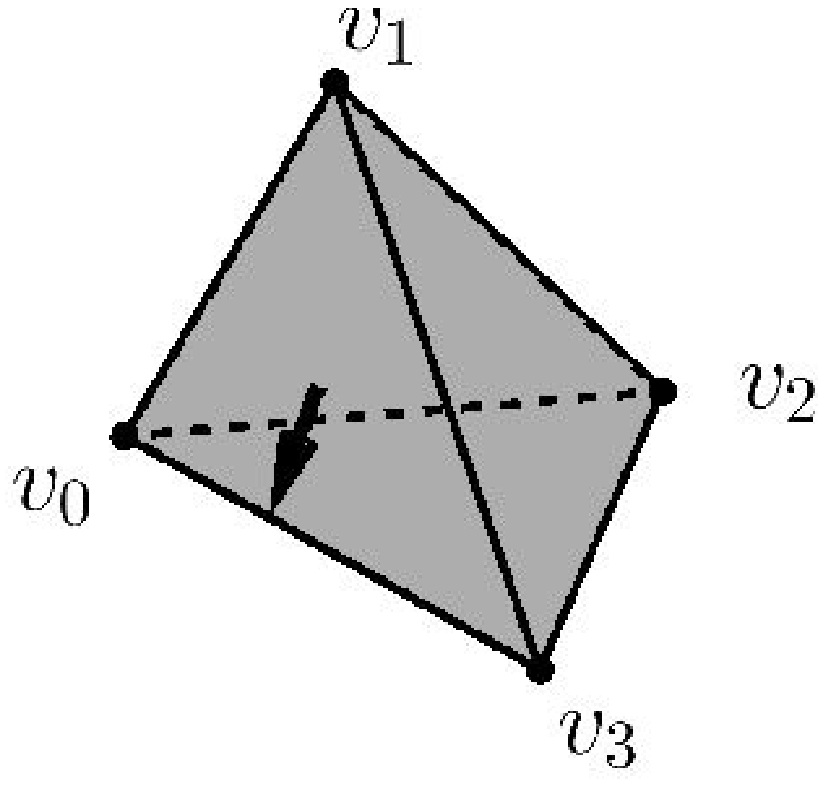}
  \end{center}
  \caption{2-complex $C^2$}
  \label{cone}
  \end{figure}

\end{document}